\documentclass[a4paper,12pt,leqno]{article}
\usepackage[nofancy]{svninfo}
\svnInfo $Id: compatible.tex 1147 2019-05-28 15:09:54Z zheng $
\usepackage[margin=3cm]{geometry}
\usepackage[all]{xy}
\usepackage{colonequals}
\usepackage{stackrel}
\usepackage{amssymb}
\usepackage{amsxtra}
\usepackage{amsthm}
\usepackage[T1]{fontenc}
\usepackage{lmodern}
\usepackage[pagebackref,breaklinks]{hyperref}
\usepackage[lite,abbrev,shortalphabetic]{amsrefs}
\usepackage[shortlabels]{enumitem}
\setlist[enumerate,1]{(1)}
\setlist{nosep}
\theoremstyle{definition}
\newtheorem{defn}[subsection]{Definition}
\newtheorem{notation}[subsection]{Notation}
\theoremstyle{plain}
\newtheorem{theorem}[subsection]{Theorem}
\newtheorem{lemma}[subsection]{Lemma}
\newtheorem{prop}[subsection]{Proposition}
\newtheorem{cor}[subsection]{Corollary}
\theoremstyle{remark}
\newtheorem{remark}[subsection]{Remark}

\newcommand{\Z}{\mathbb{Z}}
\newcommand{\F}{\mathbb{F}}
\newcommand{\Q}{\mathbb{Q}}
\newcommand{\cF}{\mathcal{F}}
\newcommand{\cG}{\mathcal{G}}
\newcommand{\cO}{\mathcal{O}}
\newcommand{\cX}{\mathcal{X}}
\newcommand{\Img}{\mathrm{Im}}

\newcommand{\Spec}{\mathrm{Spec}}
\newcommand{\Unip}{\mathrm{Unip}}
\newcommand{\GL}{\mathrm{GL}}
\newcommand{\Br}{\mathrm{Br}}
\newcommand{\Fr}{\mathrm{Fr}}
\newcommand{\cl}{\mathrm{cl}}
\newcommand{\lisse}{\mathrm{lisse}}
\newcommand{\naive}{\mathrm{naive}}
\newcommand{\qu}{\mathrm{qu}}
\newcommand{\QU}{\mathrm{QUnip}}
\newcommand{\Gal}{\mathrm{Gal}}
\newcommand{\tr}{\mathrm{tr}}
\newcommand{\Ker}{\mathrm{Ker}}

\newcommand{\gr}{\mathrm{gr}}

\newcommand{\Zlb}{\overline{\Z_\ell}}
\newcommand{\Qlb}{\overline{\Q_\ell}}
\newcommand{\Flb}{\overline{\F_\ell}}
\newcommand{\Flib}{\overline{\F_{\ell_i}}}
\newcommand{\Qli}{\Q_{\ell_i}}
\newcommand{\Qlib}{\overline{\Q_{\ell_i}}}
\newcommand{\Zlib}{\overline{\Z_{\ell_i}}}
\newcommand{\simto}{\xrightarrow{\sim}}

\numberwithin{equation}{section}

\begin{document}

\title{Compatible systems and ramification}
\author{Qing Lu\thanks{School of Mathematical Sciences, Beijing Normal University, Beijing
100875, China; School of Mathematical Sciences, University of the Chinese
Academy of Sciences, Beijing 100049, China; email: \texttt{qlu@bnu.edu.cn}.
Partially supported by National Natural Science Foundation of China Grants
11371043, 11501541.}\and Weizhe Zheng\thanks{Morningside Center of
Mathematics and Hua Loo-Keng Key Laboratory of Mathematics, Academy of
Mathematics and Systems Science, Chinese Academy of Sciences, Beijing
100190, China; University of the Chinese Academy of Sciences, Beijing
100049, China; email: \texttt{wzheng@math.ac.cn}. Partially supported by
National Natural Science Foundation of China Grants 11621061, 11688101,
11822110; National Center for Mathematics and Interdisciplinary Sciences,
Chinese Academy of Sciences.}
\thanks{
 (Secondary).}}
%\date{}
\maketitle

\begin{abstract}
We show that compatible systems of $\ell$-adic sheaves on a scheme of
finite type over the ring of integers of a local field are compatible
along the boundary up to stratification. This extends a theorem of Deligne
on curves over a finite field. As an application, we deduce the
equicharacteristic case of classical conjectures on $\ell$-independence
for proper smooth varieties over complete discrete valuation fields.
Moreover, we show that compatible systems have compatible ramification. We
also prove an analogue for integrality along the boundary.
\end{abstract}

\section{Introduction}\label{s.0}

Let $S=\Spec(\cO_K)$ be the spectrum of an excellent Henselian discrete
valuation ring $\cO_K$ of finite residue field $k=\F_q$ of characteristic
$p$. Let $K$ be the fraction field of $\cO_K$. Given a scheme $X$ of finite
type over $S$ and a prime $\ell\neq p$, we let $K(X,\Qlb)$ denote the
Grothendieck group of constructible $\Qlb$-sheaves on $X$, where $\Qlb$
denotes an algebraic closure of $\Q_\ell$. We fix a field $Q$, an index set
$I$, and for each $i\in I$, a prime number $\ell_i$ and an embedding
$\iota_i\colon Q\to {\Qlib}$. Let $\lvert X\rvert$ be the set of locally
closed points of $X$. In other words, $\lvert X\rvert =\lvert X_k\rvert\cup
\lvert X_K\rvert$ is the union of the sets of closed points of the two
fibers. Note that the residue field of $x\in \lvert X\rvert$ is a finite
extension of $k$ or $K$, and the local Weil group $W(\bar x/x)\subseteq
\Gal(\bar x/x)$ is defined for any geometric point $\bar x$ above $x$. We
say that a system $(L_i) \in \prod_{i\in I}K(X,\Qlib)$ is \emph{compatible}
if for every $\bar x$ above $x\in \lvert X\rvert$, and for every $F\in
W(\bar x/x)$, the local traces are compatible: there exists $a\in Q$ such
that $\tr(F,(L_i)_{\bar x})=\iota_i(a)$ for all $i\in I$ \cite[D\'efinition
4.13]{Zind}.

In this paper, we study the compatibility of compatible systems along the
boundary. We let $K_\lisse(X,\Qlb)$ denote the Grothendieck group of lisse
$\Qlb$-sheaves on~$X$.

\begin{defn}\label{d.1}
Let $\bar X$ be a normal scheme of finite type over $S$ and let $X$ be a
dense open subscheme. We say that $(L_i) \in \prod_{i\in
I}K_\lisse(X,\Qlib)$ is \emph{compatible on $\bar X$}, if for every $x\in
\lvert \bar X\rvert$, every geometric point $\bar a$ of $X_{(x)}\colonequals
\bar X_{(x)}\times_{\bar X} X$, and every $F\in W(X_{(x)},\bar a)$,
$(\tr(F,(L_i)_{\bar a}))_{i\in I}$ is compatible. Here $\bar X_{(x)}$
denotes the Henselization of $\bar X$ at $x$, and $W(X_{(x)},\bar a)$
denotes the Weil group, namely the inverse image of $W(\bar x/x)\subseteq
\Gal(\bar x/x)$ by the surjective homomorphism $\pi_1(X_{(x)},\bar a)\to
\pi_1(\bar X_{(x)},\bar a)\simeq \Gal(\bar x/x)$.
\end{defn}

In the case where $\bar X$ is an integral smooth curve over $k$ or $K$ and
$x\in \bar X-X$, $X_{(x)}$ is the spectrum of a field extension of the
function field $E$ of $X$ and its fundamental group is the decomposition
group of $E$ at $x$, subgroup of the Galois group of $E$.

We call $X\subseteq\bar X$ a normal compactification over $S$ if $\bar X$ is
normal, proper over~$S$, and contains $X$ as a dense open subscheme. Our
first result is that compatible systems are compatible along the boundary up
to stratification.

\begin{theorem}\label{t.main}
Let $X$ be a scheme of finite type over $S$ and let $(L_i) \in \prod_{i\in
I}K_\lisse(X,\Qlib)$ be a compatible system with $I$ finite. Then there
exists a finite stratification $X=\bigcup_\alpha X_\alpha$ by normal
subschemes such that each $X_\alpha$ admits a normal compactification $\bar
X_\alpha$ over $S$ such that $(L_i|_{X_\alpha})_{i\in I}$ is compatible on
$\bar X_{\alpha}$.
\end{theorem}

We refer to Corollary \ref{c.key} for the equivalent statement that
compatible systems are compatible along the boundary up to modification. In
the case of a curve over a finite field we recover a theorem of Deligne
\cite[Th\'eor\`eme 9.8]{Deligne} (see Corollary \ref{c.main}). Takeshi Saito
gave an example of a compatible system on a smooth surface $X$ that is not
compatible on a given smooth compactification $\bar X$ (private
communication with Hiroki Kato).

Theorem \ref{t.main} implies the following valuative criterion for
compatible systems, analogous to Gabber's valuative criterion for Vidal's
ramified part of the fundamental group \cite[Section 6.1]{Vidal2}.

\begin{cor}\label{c.val}
Let $X$ be a scheme of finite type over $S$ and let $(L_i)\in \prod_{i\in
I}K(X,\Qlib)$. Consider commutative squares of schemes
\begin{equation}\label{e.val}
\xymatrix{\eta\ar@{^{(}->}[d] \ar[r] & X\ar[d]\\
V\ar[r] & S}
\end{equation}
where $V=\Spec(\cO_L)$ with $\cO_L$ a Henselian valuation ring, and
$\eta=\Spec(L)$ is the generic point of $V$. Let $\bar \eta\to \eta$ be a
geometric point and let $t\in V$ be the closed point.
\begin{enumerate}
\item $(L_i)_{i\in I}$ is a compatible system if and only if for every
    commutative square \eqref{e.val} with $t$ quasi-finite over $S$,
    $(\tr(F,(L_i)_{\bar \eta}))_{i\in I}$ is compatible for all $F\in
    W(\bar \eta/\eta)$.
\item If $(L_i)_{i\in I}$ is a compatible system and \eqref{e.val} is a
    commutative square with $V$ strictly Henselian, then
    $(\tr(F,(L_i)_{\bar \eta}))_{i\in I}$ is compatible for all $F\in
    \Gal(\bar \eta/\eta)$.
\end{enumerate}
\end{cor}

Note that here we do not assume $\cO_L$ to be a discrete valuation ring or
that $V\to S$ is local.

As an application, we deduce the equicharacteristic case of some classical
conjectures by Serre on $\ell$-independence (Conjectures C${}_4$, C${}_5$,
and C${}_8$ of \cite[Section 2.3]{Serre}, cf.\ \cite[Appendix, Problems 1
and 2]{ST}).

\begin{theorem}\label{t.ind} Let $\cO_L$ be a Henselian discrete valuation ring of characteristic
    $p>0$, of fraction field $L$ and residue field $\kappa$.  Let $X$ be a proper smooth scheme
    over $L$. Let $\bar L$ be a separable closure of
    $L$ and let $\bar \kappa$ be the residue field of $\bar L$. Let $X_{\bar L}=X\otimes_L \bar L$.
\begin{enumerate}
\item For each $m$ and each $F\in I(\bar L/L)\colonequals\Ker(\Gal(\bar
    L/L)\to \Gal(\bar \kappa/\kappa))$, $\tr(F,H^m(X_{\bar L},\Q_\ell))$
    is a rational integer independent of $\ell\neq p$.
\item (cf.\ \cite[Theorem 3.3]{Terasoma}) Assume that $\kappa$ is a finite
    field. Then for each $m$, each~$i$, and each $F\in W(\bar L/L)$ whose
    image in $W(\bar\kappa/\kappa)$ is the $n$-th power of the geometric
    Frobenius for $n\ge 0$, we have
\begin{itemize}
\item[(2a)] $\tr(F,\gr^M_i H^m(X_{\bar L},\Q_\ell))$ is a rational
    integer independent of $\ell\neq p$, where $M$ denotes the monodromy
    filtration; in particular,

\item[(2b)] $\tr(F, H^m(X_{\bar L},\Q_\ell))$ is a rational integer
    independent of $\ell\neq p$.
\end{itemize}
\end{enumerate}
\end{theorem}

Part (2) was claimed in \cite[Theorem 6.1]{CL}, but the proof given there is
incomplete\footnote{The authors of \cite{CL} have been made aware of this
and have submitted a corrigendum.}. A weaker form of (2) was proved by
Terasoma \cite[Theorem 3.3]{Terasoma}.

\begin{remark}\label{r.mw}
Theorem \ref{t.ind} (1) is the equicharacteristic $p>0$ case of Serre's
Conjecture C${}_4$. Theorem \ref{t.ind} (2a) implies the equicharacteristic
case of Conjecture C${}_5$ (Remark \ref{r.ind} (2)), while (2b) implies the
equicharacteristic case of Conjecture C${}_8$. Parts (1) and (2b) of Theorem
\ref{t.ind} hold more generally over a Henselian valuation field of
characteristic $p>0$ without assuming that the valuation is discrete (Remark
\ref{r.ind} (3)).

The alternating sum $\sum_m (-1)^m \tr(F,H^m(X_{\bar L},\Q_\ell))$ of the
traces in (1) and (2b) was known to be a rational integer independent of
$\ell\neq p$ more generally for $X$ separated of finite type over $L$
without the equicharacteristic assumption. See Vidal \cite[Proposition
4.2]{Vidal1} (combined with Laumon \cite[Th\'eor\`eme 1.1]{Laumon}), Ochiai
\cite{Ochiai}, \cite{Zint} and \cite{Zind} (Theorems \ref{t.six} and
\ref{t.intf} below). Our valuative criterion allows to further extend the
results on the alternating sum to Henselian valuation fields \cite{LZ}.
\end{remark}

In the case where $X$ is defined over a curve over a finite field, Theorem
\ref{t.ind} (2b) follows from Deligne's theorem for curves mentioned above
and results of Weil II \cite{WeilII}. In the general case, after spreading
out, the base becomes a variety over a finite field and we apply Corollary
\ref{c.val}.

In Section~\ref{s.1}, we give the proofs of Theorems \ref{t.main} and
\ref{t.ind}. The proof of Theorem \ref{t.main} relies on the preservation of
compatible systems under direct images \cite[Proposition 4.15]{Zind}. Over a
finite field the latter is a theorem of Gabber \cite[Theorem~2]{Gabber}.

In Section~\ref{s.3}, we study integrality along the boundary and prove an
analogue of Theorem \ref{t.main}. This generalizes a theorem of Deligne on
non-Archimedean absolute values of liftings of local Frobenius for curves
over finite fields \cite[Th\'eor\`eme 1.10.3]{WeilII}.

Our original motivation for studying compatibility along the boundary is to
understand the relationship between compatible systems of $\Qlb$-sheaves and
systems of $\Flb$-sheaves with compatible wild ramification. The latter and
variants were studied in recent work of Saito, Yatagawa (\cite{SY},
\cite{Yatagawa}) and Guo \cite{Guo}, generalizing earlier work of Deligne
\cite{IllBrauer} and Vidal \cites{Vidal1, Vidal2}. In Section~\ref{s.2}, we
deduce from our valuative criterion for compatible systems that compatible
systems have compatible ramification, and consequently, their reductions
have compatible wild ramification. These notions are defined using Vidal's
ramified part of the fundamental group, which involve images of local
inertia groups at geometric points of compactifications $\bar X$ of $X$. We
define the decomposed part of the fundamental group by taking instead images
of the local decomposition groups at $x\in \lvert\bar X\rvert$. We show, as
another application of Theorem \ref{t.main}, that the union of the images of
the local decomposition (or Weil) groups for $x\in \lvert X\rvert$ is dense
in the decomposed part.

\paragraph*{Acknowledgment} We thank H\'el\`ene Esnault, Yongquan Hu, and Luc Illusie for many useful
discussions. We thank the referee for a careful reading of the manuscript
and for many helpful suggestions.

\section{Compatible systems along the boundary}\label{s.1}

The strategy of the proof of Theorem \ref{t.main} is to reduce to the case
of lisse sheaves tamely ramified along a normal crossing divisor with
unipotent local monodromy. For this we need to work with finite group
actions. We now review the notion of compatible systems on Deligne-Mumford
stacks \cite[Section~5]{Zind} and finite quotient stacks in particular. In
this paper, Deligne-Mumford stacks are assumed to be quasi-separated with
separated diagonal.

Let $\bar k$ be a separable closure of $k$. Each $F\in W(\bar k/k)$ is the
$n$-th power of the geometric Frobenius $\Fr\colon a\mapsto a^{1/q}$ for
some $n\in \Z$. We call $n$ the \emph{degree} of $F$. For an integer $N$, we
let $W^{\ge N}(\bar k/k)$ denote the subset $\{\Fr^n\mid n\ge N\}$.

\begin{notation}\label{n.Y}
For any connected Deligne-Mumford stack $Y$ over $S$ and any geometric point
$\bar a\to Y$, we define the Weil group $W(Y,\bar a)$ to be the inverse
image of the Weil group $W(\bar k/k)$ by the homomorphism
\[r\colon \pi_1(Y,\bar a)\to \pi_1(S,\bar a)\simeq \Gal(\bar k/k).\]
We define the \emph{degree} of $F\in W(Y,\bar a)$ to be the degree of
$r(F)$. We let $W^{\ge N}(Y,\bar a)$ denote the subset $r^{-1}(W^{\ge
N}(\bar k/k))$ of elements of degree $\ge N$.
\end{notation}

Let $X$ be a Deligne-Mumford stack.  For a point $\xi$ of $X$, we let
$X_\xi$ denote the residual gerbe, which is necessarily a quotient stack
$[x/G]$ by a finite group $G$ of the spectrum of a field $x$ (cf.\
\cite[page 13]{IZ}). For a geometric point $\bar x$ above $x$, we have
\[\pi_1([x/G],\bar x)\simeq \Gal(\bar x/y)\times_{\Gal(x/y)} G,\]
where $y=x/G$.

Assume $X$ of finite type over $S$. We let $\lvert X\rvert$ denote the set
of locally closed points of $X$. For $\xi\in\lvert X\rvert$, $x$ is
quasi-finite over $S$, spectrum of a finite field extension of $k$ or $K$.
The Weil group $W([x/G],\bar x)\subseteq \pi_1([x/G],\bar x)$ is the inverse
image of the Weil group $W(\bar x/y)\subseteq \Gal(\bar x/y)$ by the
homomorphism $\pi_1([x/G],\bar x)\to \Gal(\bar x/y)$, which is surjective of
kernel the inertia group.

\begin{defn}\label{d.stack}
We say that $(L_i) \in \prod_{i\in I}K(X,\Qlib)$ is \emph{compatible} if it
satisfies the following equivalent conditions.
\begin{enumerate}
\item For every $\xi\in \lvert X\rvert$, every geometric point $\bar x$
    above $\xi$, and every $F\in W(X_\xi,\bar x)$, $(\tr(F,(L_i)_{\bar
    x}))_{i\in I}$ is compatible.

\item For every quasi-finite morphism $f\colon x\to X$ where $x$ is the
    spectrum of a field, $(f^*L_i)_{i\in I}$ is a compatible system on $x$
    (Section \ref{s.0}).

\item For every smooth morphism $f\colon Y\to X$ of finite type with $Y$ a
    scheme, $(f^*L_i)_{i\in I}$ is a compatible system on $Y$ (Section
    \ref{s.0}).
\end{enumerate}
\end{defn}

The implications (1) $\Rightarrow$ (2) $\Rightarrow$ (3) are trivial. (3)
$\Rightarrow$ (2) follows from the existence of smooth neighborhoods
\cite[Th\'eo\`eme 6.3]{LMB}. (2) $\Rightarrow$ (1) follows from
\cite[Proposition 5.6]{Zind} applied to the quotient stack $X_\xi$, which is
based on a method of Deligne and Lusztig \cite[proof of Proposition
3.3]{DL}.

In the case where $X=[Y/G]$ is a quotient stack of a scheme $Y$ by a finite
group, the residual gerbe at the image of $y\in Y$ is $[y/D(y)]$, where
$D(y)<G$ is the decomposition group.

The main result of \cite{Zind} can be stated as follows.

\begin{theorem}\label{t.six}
Compatible systems on Deligne-Mumford stacks of finite type over~$S$ are
stable under Grothendieck's six operations and duality.
\end{theorem}

This is stated for Deligne-Mumford stacks of finite type over $k$ or $K$ in
\cite[Proposition 5.8]{Zind}, but the same proof applies over $S$ with
\cite[Th\'eor\`eme 1.16]{Zind} replaced by the more general
\cite[Proposition 4.15]{Zind}. The case of schemes of finite type over $k$
is a theorem of Gabber \cite[Theorem~2]{Gabber}.

We will only need the stability under $Rj_*$ for an open immersion $j$.

\begin{remark}\label{r.rat}
Let $x$ be quasi-finite over $S$ and let $(L_i)\in \prod_{i\in I}
K(x,\Qlib)$. If there exists an integer $N$ such that $(\tr(F,(L_i)_{\bar
x}))_{i\in I}$ is compatible for all $F\in W^{\ge N}(x,\bar x)$, then the
same holds for all $F\in W(x,\bar x)$ by \cite[Proposition 1.15]{Zind}
(consequence of Grothendieck's arithmetic local monodromy theorem
\cite[Appendix]{ST} and a rationality lemma \cite[Lemma 8.1]{IllMisc}).
\end{remark}

In the regular case, compatibility of systems of unramified lisse sheaves
extends to the boundary by the following variant of \cite[Proposition
3.10]{Zind}.

\begin{prop}\label{p.reg}
Let $X$ be a regular Deligne-Mumford stack of finite type over~$S$ and let
$(L_i)\in \prod_{i\in I}K_\lisse(X,\Qlib)$. Assume that $(L_i|_U)_{i\in I}$
is compatible for some dense open substack $U\subseteq X$. Then $(L_i)_{i\in
I}$ is compatible.
\end{prop}

\begin{proof}
The proof is very similar to that of \cite[Proposition 3.10]{Zind}. A
related argument will be used in the proof of Proposition \ref{p.NCD} below.
By induction, we may assume that $D=X-U$ is regular and purely of
codimension $d\ge 1$. Let $j\colon U\to X$ be the open immersion. By Theorem
\ref{t.six}, $(Rj_*(L_i|_U))_{i\in I}$ is compatible. By projection formula,
\[L_i \otimes_{\Qli} Rj_* \Qli\simeq Rj_*(L_i|_U).\]

Gabber's absolute purity theorem (\cite[Theorem 2.1.1]{Gabberpure},
\cite[Th\'eo\`eme 3.1.1]{Riou}) extends to Deligne-Mumford stacks: the
refined cycle class $\cl_f\in H^{2d}_D(X,\Q_\ell(d))$ induces an isomorphism
$\Q_\ell\simto Rf^!\Q_\ell(d)[2d]$, where $f\colon D\to X$ denotes the
closed immersion. Indeed, the definition of $\cl_f$ \cite[Definition
1.1.2]{Gabberpure} holds without change (with Chern classes defined by
Grothendieck \cite[Section 1]{Grothendieck}) and the fact that it induces an
isomorphism reduces to the case of schemes. It follows that we have
\[R^m j_* \Q_\ell\simeq \begin{cases}
  \Q_\ell&m=0,\\
  (\Q_\ell)_D(-d)&m=2d-1,\\
  0&\text{otherwise.}
\end{cases}
\]
Alternatively we can reduce Proposition \ref{p.reg} to the case of schemes
using Definition \ref{d.stack} (3).

For $x\to D$ quasi-finite and $F\in W(\bar x/x)$ of degree $n$,
$\tr(F,(Rj_*\Q_\ell)_{\bar x})=1-q^{nd}$. Thus, for $n\neq 0$,
$\tr(F,(L_i)_{\bar x})$ can be recovered from $\tr(F,(Rj_*(L_i|_U))_{\bar
x})$. Therefore, $(L_i)_{i\in I}$ is compatible by Remark \ref{r.rat}.
\end{proof}

Next we define compatibility on the boundary in the equivariant setting. Let
$\bar X$ be a scheme equipped with the action of a finite group $G$. For
$x\in \bar X$, the decomposition group $D(x)$ acts on $\bar X_{(x)}$. For
any geometric point $\bar x$ above $x$, we have $\pi_1([x/D(x)],\bar
x)\simeq \pi_1([\bar X_{(x)}/D(x)],\bar x)$. For $\bar X$ normal,
$X\subseteq \bar X$ a $G$-stable dense open subscheme, and $\bar a\to
X_{(x)}$ a geometric point, the homomorphism
\begin{equation}\label{e.surj}
\pi_1([X_{(x)}/D(x)],\bar a)\to \pi_1([\bar X_{(x)}/D(x)],\bar a)\simeq \pi_1([x/D(x)],\bar
x)
\end{equation}
is surjective.

\begin{defn}\label{d.compb}
Let $\bar X$ be a normal scheme of finite type over $S$ equipped with an
action of $G$ by $S$-automorphisms and let $X$ be a $G$-stable dense open
subscheme. We say that $(L_i)\in \prod_{i\in I}K_\lisse([X/G],\Qlib)$ is
\emph{compatible on $[\bar X/G]$} if for every $x\in \lvert \bar X\rvert$,
every geometric point $\bar a\to X_{(x)}$, and every $F\in
W([X_{(x)}/D(x)],\bar a)$, $(\tr(F,(L_i)_{\bar a}))_{i\in I}$ is compatible.
\end{defn}

\begin{remark}\leavevmode\label{r.triv}
\begin{enumerate}
\item $(L_i)\in \prod_{i\in I}K_\lisse([X/G],\Qlib)$ is compatible on
    $[X/G]$ in the sense of Definition \ref{d.compb} if and only if it is
    compatible in the sense of Definition \ref{d.stack}. This follows from
    the isomorphism in \eqref{e.surj}.
\item Let $U\subseteq X\subseteq \bar X$ be a $G$-stable dense open
    subscheme. Then $(L_i)\in \prod_{i\in I}K_\lisse([X/G],\Qlib)$ is
    compatible on $[\bar X/G]$ if and only if $(L_i|_{[U/G]})_{i\in I}$ is
    compatible on $[\bar X/G]$. This follows from the fact for $x\in \bar
    X$, the homomorphism $\pi_1([U_{(x)}/D(x)],\bar a)\to
    \pi_1([X_{(x)}/D(x)],\bar a)$ is surjective.
\item Assume that $G$ acts freely on $X$. Let $Y=X/G$ and $\bar Y=\bar
    X/G$ be the quotient spaces. Then, for all $x\in \bar X$, if $y\in
    \bar Y$ denotes its image, then $[X_{(x)}/D(x)]\simeq Y_{(y)}$. Thus,
    in this case, $(L_i)_{i\in I}$ on $X/G$ is compatible on $[\bar X/G]$
    if and only if it is compatible on $\bar X/G$.
\end{enumerate}
\end{remark}

\begin{remark}\label{r.closed}
Let $x\in \lvert \bar X\rvert$ be a point that is not closed. The closure
$Y=\overline{\{x\}}\subseteq \bar X$ admits a Zariski open cover by schemes
finite over~$S$. Thus $Y=\bigcup_y Y_{(y)}$, $y$ running through closed
points of $Y$. We have $x\to Y_{(y)}\to \bar X_{(y)}$, which induces a
morphism $X_{(x)} \to X_{(y)}$. If $\bar X$ is separated, then $Y=Y_{(y)}$
and $D(x)<D(y)$. Thus in Definition \ref{d.compb}, if $\bar X$ is separated
or $G=\{1\}$, then we may restrict to closed points of $\bar X$.
\end{remark}

\begin{remark}
Given a point $\xi$ of a Deligne-Mumford stack $Y$, one can define the
Henselization of $Y$ at $\xi$ to be the limit of Deligne-Mumford stacks $V$
for decompositions of the residual gerbe $Y_\xi\to Y$ into $Y_\xi\to
V\xrightarrow{\phi} Y$ with $\phi$ representable and \'etale. Using
\cite[Lemma 3.5]{IZ}, one can show that the Henselization of $[\bar X/G]$ at
the image of $x\in \bar X$ is $[\bar X_{(x)}/D(x)]$. Thus Definition
\ref{d.compb} depends only on the quotient stacks and can be extended to
Deligne-Mumford stacks.
\end{remark}

Let $\bar X$ be a regular Deligne-Mumford stack and let $D\subseteq \bar X$
be a normal crossing divisor. We say that a lisse $\Qlb$-sheaf $\cF$ on
$X=\bar X-D$ is \emph{tamely ramified on $\bar X$} if for every geometric
point $\bar x$ above a generic point of $D$ and every geometric point $\bar
a$ of $X_{(\bar x)}\colonequals \bar X_{(\bar x)}\times_{\bar X}X$, the wild
inertia group of $X_{(\bar x)}$ acts trivially on $\cF_{\bar a}$. Here $\bar
X_{(\bar x)}$ denotes the strict Henselization. We say that $L\in
K_\lisse(X,\Qlb)$ is \emph{tamely ramified on $\bar X$} if $L=[\cF]-[\cG]$
with $\cF$ and $\cG$ lisse and tamely ramified on $\bar X$.

\begin{prop}\label{p.NCD}
Let $\bar X$ be a regular scheme of finite type over~$S$ equipped with an
action of a finite group $G$ by $S$-automorphisms. Let $D\subseteq \bar X$
be a normal crossing divisor such that $X=\bar X-D$ is $G$-stable. Let
$(L_i)\in \prod_{i\in I}K_\lisse([X/G],\Qlib)$ be a compatible system.
Assume that one of the following conditions holds:
\begin{enumerate}
\item For each $i$ there exist lisse $\Zlib$-sheaves $\cF_i$ and $\cG_i$
    such that $L_i|_X=([\cF_i] -[\cG_i])\otimes_{\Zlib} \Qlib$, and
    $\cF_i\otimes_{\Zlib} \Zlib/\ell_i^{c}\Zlib$ and $\cG_i\otimes_{\Zlib}
    \Zlib/\ell_i^{c}\Zlib$ are constant for some rational number
    $c>\frac{1}{\ell_i-1}$. Here $\Zlib$ denotes the ring of integers of
    $\Qlib$.
\item $G=\{1\}$ and each $L_i$ is tamely ramified on $\bar X$.
\end{enumerate}
Then $(L_i)_{i\in I}$ is compatible on $[\bar X/G]$.
\end{prop}

For the proof of Theorem \ref{t.main}, we will only need part (1). For the
proof of Proposition \ref{p.NCD}, we need a variant of Grothendieck's
arithmetic local monodromy theorem \cite[Appendix]{ST}. We say that a family
of matrices $\rho\colon E\to \GL_n(\Qlb)$ is \emph{quasi-unipotent} if each
$\rho(g)$, $g\in E$ is quasi-unipotent. By Remark \ref{r.unip} below, a
continuous representation $\rho\colon I\to \GL_n(\Qlb)$ of a profinite group
$I$ is quasi-unipotent if and only if $\rho$ is unipotent on an open
subgroup $I_0<I$.

\begin{remark}\label{r.unip}
Let $P_\ell^n\simeq \Qlb^n$ be the space of monic polynomials of degree $n$.
The subset $P_\ell^{n,\qu}\subseteq P_\ell^n$ of polynomials whose roots are
roots of unity is discrete and closed. This follows from continuity of roots
and the fact that the only root of unity in $1+\ell^c\Zlb$ is $1$, where
$c>\frac{1}{\ell-1}$ is any rational number.

The function $M_n(\Qlb)\to P_\ell^n$ carrying an $n\times n$ matrix to its
characteristic polynomial is continuous. It follows that the subset
$\QU_n(\Qlb)\subseteq M_n(\Qlb)$ of quasi-unipotent matrices is closed, and
the subgroup of unipotent matrices $\Unip_n(\Qlb)<\QU_n(\Qlb)$ is open.
\end{remark}

\begin{lemma}\label{l.mono}
Consider short exact sequences of profinite groups
\[1\to P\to I\to I_\ell\to 1,\quad 1\to I\to G\to G_s\to 1,\]
with $P$ of supernatural order prime to $\ell$ and $I_\ell$ pro-$\ell$
Abelian. Assume that the conjugation action of $G_s$ on $I_\ell$ is given by
a character $\chi\colon G_s\to \Z_\ell^\times$ of infinite order. Then any
continuous representation $\rho\colon G\to \GL_n(\Qlb)$ is quasi-unipotent
on~$I$. Moreover, for any rational number $c>\frac{1}{\ell-1}$, we have
\[\QU_n(\Qlb)\cap (1+\ell^cM_n(\Zlb))\subseteq \Unip_n(\Qlb).\]
\end{lemma}

\begin{proof}
Let $U\in \QU_n(\Qlb)\cap (1+\ell^cM_n(\Zlb))$. Then $U^a$ is unipotent for
some integer $a>0$, and $\log(U)=\frac{1}{a}\log(U^a)$ is nilpotent, so that
$U=\exp(\log(U))$ is unipotent. This proves the second assertion.

The proof of the first assertion is identical to that of Grothendieck. Up to
replacing $G$ by an open subgroup, we may assume that $\rho$ factors through
the open subgroup $1+\ell^c M_n(\Zlb)$. Then $\rho(P)=1$. Take $g\in G_s$
such that $\chi(g)$ is not a root of unity. For $t\in I$, $\rho(t)$ is
conjugate to $\rho(t)^{\chi(g)}$, so that $M=\log(\rho(t))$ is conjugate to
$\log(\rho(t)^{\chi(g)})=\chi(g)\log(\rho(t))=\chi(g)M$. Thus
$\chi(g)^m\tr(M^m)=\tr(M^m)$, so that $\tr(M^m)=0$ for all $m\ge 1$.
Therefore, $M$ is nilpotent and $\rho(t)=\exp(M)$ is unipotent.
\end{proof}

\begin{proof}[Proof of Proposition \ref{p.NCD}]
The proof is similar to a part of Deligne's proof of \cite[Th\'eor\`eme
9.8]{Deligne}.

We may assume that the index set $I$ is finite. Let $L_i=[\cF_i]-[\cG_i]$
for $\cF_i$ and $\cG_i$ lisse on $[X/G]$. Let $x\in \lvert D\rvert$. Lemma
\ref{l.mono} applies to the tame fundamental group
$\pi_1^t([X_{(x)}/D(x)],\bar a)$ (cf.\ \cite[1.7.12.1]{WeilII} in the case
$x$ above $k$). Indeed, by Abhyankar's lemma \cite[XIII Corollaire
5.3]{SGA1}, we have a short exact sequence
\[1\to I_t\to \pi_1^t([X_{(x)}/D(x)],\bar a) \xrightarrow{r} \pi_1([x/D(x)],\bar x) \to 1,\]
where $I_t=\prod_\ell\Z_\ell(1)^d$, $d$ is the number of irreducible
components of $D\times_{\bar X} \bar X_{(x)}$ and $\ell$ runs through primes
different from the characteristic of $x$. Note that for any semisimple
continuous representation $\rho\colon \pi_1^t([X_{(x)}/D(x)],\bar a)\to
\GL_n(\Qlb)$ and any $g\in I_t$ with $\rho(g)$ unipotent, we have
$\rho(g)=1$. Indeed, on each graded piece of the monodromy filtration given
by the nilpotent operator $\log(\rho(g))$, $g$ acts by $1$.

In case (1) $I_t$ acts unipotently on $(\cF_i)_{\bar a}$ and $(\cG_i)_{\bar
a}$. In case (2) there exists an open subgroup $I'_t$ of $I_t$ that acts
unipotently on $(\cF_i)_{\bar a}$ and $(\cG_i)_{\bar a}$. Each $g\in
W(X_{(x)},\bar a)$ of degree $0$ acts quasi-unipotently on $(\cF_i)_{\bar
a}$ and $(\cG_i)_{\bar a}$. As in the proof of \cite[Proposition
1.15]{Zind}, there exists a subgroup $G<W(X_{(x)},\bar a)$ of finite index
such that the action of $g$ commutes with that of $G$ up to
semisimplification. By \cite[Lemma 8.1]{IllMisc}, it suffices to consider
$F\in W(X_{(x)},\bar a)$ of degree $\neq 0$. There exists an open subgroup
$H<\pi_1^t(X_{(x)},\bar a)$ containing the image of $F$ such that $H\cap
I_t\subseteq I'_t$. We may further assume that $H\cap I_t$ has the form
$NI_t$ for an integer $N>0$ invertible on $x$. Let $\bar Y_{(y)}$ be the
normalization of $\bar X_{(x)}$ in the pointed finite \'etale cover
$(Y_{(y)},\bar b)$ of $(X_{(x)},\bar a)$ corresponding to $H$. Then $F\in
W(Y_{(y)},\bar b)$. Moreover, $\bar Y_{(y)}$ is regular and the inverse
image of $D$ is a normal crossing divisor. Indeed, if $\bar X_{(\bar x)}$
and $\bar Y_{(\bar y)}$ denote the strict Henselizations and the irreducible
components of $D\times_{\bar X} \bar{X}_{(\bar x)}$ are defined by
$t_1,\dots, t_d$, then $\bar Y_{(\bar y)}\simeq \bar X_{(\bar
x)}[t_1^{1/N},\dots, t_d^{1/N}]$. Therefore, up to replacing $\bar X$ by
$\bar Y$ quasi-finite over $\bar X$ giving rise to $\bar Y_{(y)}$, we may
assume that $I_t$ acts unipotently on $(\cF_i)_{\bar a}$ and $(\cG_i)_{\bar
a}$.

Then the semisimplifications of $\cF_i|_{[X_{(x)}/D(x)]}$ and
$\cG_i|_{[X_{(x)}/D(x)]}$ factor through $r$, so that
$L_i|_{[X_{(x)}/D(x)]}$ is the pullback of $M_i\in K(\xi,\Qlb)$ via $r$,
where $\xi=[x/D(x)]$. Let $j\colon [X/G]\to [\bar X/G]$ be the open
immersion. By Theorem \ref{t.six}, $(Rj_*L_i)_{i\in I}$ is compatible. By
projection formula,
\[M_i\otimes_{\Qli} (Rj_*\Qli)_\xi\simeq (Rj_*L_i)_\xi.\]
Gabber's absolute purity theorem, extended to Deligne-Mumford stacks in the
proof of Proposition \ref{p.reg}, implies (see \cite[Theorem 7.2]{IllAst},
\cite[Corollaire 3.1.4]{Riou})
\[(R^m j_*\Q_\ell)_\xi\simeq \begin{cases}
\Q_\ell(-m)^{\binom{d}{m}}& 0\le m\le d,\\
0&\text{otherwise}.
\end{cases}
\]
Thus for $F\in W(\xi,\bar x)$ of degree $n$,
\[\tr(F,(R j_*\Q_\ell)_{\bar x})=(1-q^n)^d.\]
It follows that for $n\neq 0$, $\tr(F,(M_i)_{\bar x})$ can be recovered from
$\tr(F,(Rj_* L_i)_{\bar x})$. Therefore, $(M_i)_{i\in I}$ is compatible by
Remark \ref{r.rat}.
\end{proof}

\begin{prop}\label{p.key}
Let $X$ be an integral normal scheme separated of finite type over~$S$ and
let $(L_i)\in \prod_{i\in I}K_\lisse(X,\Qlib)$ be a compatible system with
$I$ finite. Then there exist a proper morphism $f\colon X'\to X$ with $X'$
integral normal inducing a universal homeomorphism $f^{-1}(U)\to U$ for some
nonempty open $U\subseteq X$, and a normal compactification $X'\subseteq
\bar X'$ over $S$, such that $(f^*L_i)_{i\in I}$ is compatible on $\bar X'$.
\end{prop}

\begin{proof}
We write $L_i=([\cF_i]-[\cG_i])\otimes_{\Zlib}\Qlib$. There exists a
connected finite \'etale cover $Y\to X$, Galois of group $G$, such that
$\cF_i\otimes_{\Zlib} \Zlib/2\ell\Zlib$ and $\cG_i\otimes_{\Zlib}
\Zlib/2\ell\Zlib$ are constant.

Let $S_0$ be the closed point of $S$ if $X_K$ is empty and $S$ otherwise. We
apply Gabber's refinement of de Jong's equivariant alterations \cite{dJ2} in
the form of \cite[Lemme 3.8]{Zind}, to the $G$-equivariant morphism $Y\to
T$, where $T$ is the normalization of $S_0$ in $Y$. There exists a Galois
alteration $(Z,H)\to (Y,G)$ and an $H$-equivariant open immersion
$Z\subseteq \bar Z$ with $\bar Z$ regular and projective over $S$. Moreover,
there exists an $H$-stable open subscheme $V\subseteq Z$ whose complement in
$\bar Z$ is a normal crossing divisor. Let $f\colon X'\colonequals Z/H\to
Y/G\simeq X$. By the definition of Galois alteration, there exists a
nonempty $H$-stable affine open subscheme $V_0\subseteq V$ on which $H$ acts
freely and a nonempty open subscheme $U\subseteq X$ such that $f$ induces a
universal homeomorphism $f^{-1}(U)\to U$.

By Proposition \ref{p.NCD}, $(L_i|_{[V/H]})_{i\in I}$ is compatible on
$[\bar Z/H]$. By Remark \ref{r.triv} (3), $(L_i|_{V_0/H})_{i\in I}$ is
compatible on $\bar X'\colonequals \bar Z/H$, and the proposition follows.
\end{proof}

\begin{lemma}\label{l.uh}
Let $f\colon Y\to X$ be a universal homeomorphism between normal schemes
separated of finite type over a Noetherian Nagata scheme $T$. Then for any
normal compactification $\bar Y$ of $Y$ over $T$, there exists a commutative
diagram over $T$
\[
\xymatrix{Y\ar@{^{(}->}[r]\ar[d]_{f} &\bar Y\ar[d]^{\bar f}\\
X\ar@{^{(}->}[r] & \bar X,}
\]
where $\bar X$ is a normal compactification of $X$ over $T$ and $\bar f$ is
a universal homeomorphism identifying $\bar Y$ with the normalization of
$\bar X$ in $Y$.
\end{lemma}

\begin{proof}
We may assume $X$ connected and that $f$ is not an isomorphism. Let
$K(X)\subseteq K(Y)$ be the fraction fields. There exists $n$ such that
$K(Y)^{p^n}\subseteq K(X)$, where $p>0$ is the characteristic of $K(X)$. Up
to replacing $T$ by a closed subscheme, we may assume $X\to T$ dominant. The
$n$-th relative Frobenius factors as $Y\xrightarrow{f} X\to Y^{(p^n)}$. We
take $\bar X$ to be the normalization of $\bar Y^{(p^n)}$ in $X$. The
morphism $\bar f\colon \bar Y\to \bar X$ is finite, surjective, and
radicial, hence a universal homeomorphism.
\end{proof}

\begin{proof}[Proof of Theorem \ref{t.main}]
We may assume $X$ reduced. By Proposition \ref{p.key} and Lemma \ref{l.uh},
there exist an integral normal open subscheme $X_0\subseteq X$ and a normal
compactification $X_0\subseteq \bar X_0$ such that $(L_i|_{X_0})_{i\in I}$
is compatible on $\bar X_0$. We conclude by Noetherian induction.
\end{proof}

The theorem takes the following form in the case of curves, which is a
theorem of Deligne \cite[Th\'eor\`eme 9.8]{Deligne} in the case of curves
over finite fields.

\begin{cor}\label{c.main}
Let $\bar X$ be a smooth curve over $k$ or $K$ and let $X\subseteq \bar X$
be a dense open subscheme. Then any compatible system $(L_i) \in \prod_{i\in
I}K_\lisse(X,\Qlib)$ is compatible on $\bar X$.
\end{cor}

In this case, one may also directly adapt the proof of Proposition
\ref{p.NCD} with $\pi_1^t$ replaced by $\pi_1$.

\begin{remark}\label{r.comp}
Every pair of compactifications $Y\subseteq \bar Y_1$ and $Y\subseteq \bar
Y_2$ over $S$ (inclusions of dense open subschemes with $\bar Y_1$ and $\bar
Y_2$ proper over $S$) are dominated by a third one: there exist a
compactification $Y\subseteq \bar Y$ over $S$ and morphisms $\bar Y\to \bar
Y_1$ and $\bar Y\to \bar Y_2$ over $S$ inducing the identity on $Y$. It
suffices to take $\bar Y$ to be the closure of the diagonal embedding
$Y\subseteq \bar Y_1\times_S \bar Y_2$. In the case where $Y$ is normal, we
may even take $\bar Y$ to be normal by normalization.
\end{remark}

It follows that in the situation of Theorem \ref{t.main}, every
compactification $X_\alpha\subseteq \bar X'_\alpha$ is dominated by a normal
compactification $X_\alpha\subseteq \bar X''_\alpha$ such that
$(L_i|_{X_\alpha})_{i\in I}$ is compatible on $\bar X''_{\alpha}$. This
implies the following refinement of Proposition \ref{p.key}, which says that
compatible systems are compatible along the boundary up to modification.

\begin{cor}\label{c.key}
Let $\bar X$ be a reduced scheme separated of finite type over $S$ and let
$X\subseteq \bar X$ be a dense open subscheme. Let $(L_i)\in \prod_{i\in
I}K_\lisse(X,\Qlib)$ be a compatible system with $I$ finite. Then there
exists a proper birational morphism $f\colon \bar X'\to \bar X$ with $\bar
X'$ normal such that $(f_X^*L_i)_{i\in I}$ is compatible on $\bar X'$. Here
$f_X\colon f^{-1}(X)\to X$ is the restriction of $f$.
\end{cor}

\begin{proof}
Up to replacing $\bar X$ by a compactification, we may assume $\bar X$
proper over $S$. By Theorem \ref{t.main}, there exist a dense open subscheme
$U$ and a normal compactification $U\subseteq \bar U$ such that
$(L_i|_U)_{i\in I}$ is compatible on $\bar U$. Let $U\subseteq \bar X'$ be a
normal compactification dominating $U\subseteq \bar X$ and $U\subseteq \bar
U$ and let $f\colon \bar X'\to \bar X$ be the morphism. Then
$((f_X^*L_i)|_U)_{i\in I}$ is compatible on $\bar X'$. We conclude by Remark
\ref{r.triv} (2).
\end{proof}

\begin{proof}[Proof of Corollary \ref{c.val}]
The ``if'' part of (1) follows from the definition. We prove (2) and the
``only if'' part of (1). We may assume $I$ finite. Up to replacing $X$ by
the closure of the image $\tau\in X$ of $\eta$, we may assume that $X$ is
irreducible of generic point $\tau$. Up to shrinking $X$, we may assume $X$
separated and $L_i\in K_\lisse(X,\Qlib)$ for all $i$. Let $X\subseteq \bar
X$ be a compactification over $S$. We apply Corollary \ref{c.key}. Let
$X'=f^{-1}(X)$. Note that \eqref{e.val} gives rise to a commutative square
\[\xymatrix{\eta\ar@{^{(}->}[d] \ar[r] & \bar X'\ar[d]\\
V\ar@{-->}[ru]^g\ar[r] & S.}
\]
By the valuative criterion of properness, there exists a slashed arrow $g$
as indicated, making the diagram commutative. In case (1), $g$ induces
$\eta\to X'_{(x)}$ and $W(\bar \eta/\eta)\to W(X'_{(x)},\bar \eta)$, where
$x=g(t)$. In case (2), $g$ induces $\eta\to X'_{(g(\bar t))}\colonequals
\bar X'_{(g(\bar t))}\times_{\bar X'}X'$, where $\bar X'_{(g(\bar t))}$
denotes the strict Henselization of $\bar X'$ at the geometric point $g(\bar
t)$, image of the geometric point $\bar t=t$ of $V$ under $g$. The geometric
point $g(\bar t)$ of $\bar X'$ specializes to a geometric point $\bar x$
above $x\in \lvert \bar X'\rvert$. We have $\eta\to X'_{(g(\bar t))}\to
X'_{(\bar x)}\to X'_{(x)}$, which induces $\Gal(\bar \eta/\eta)\to
\pi_1(X'_{(\bar x)},\bar \eta) \subseteq W(X'_{(x)},\bar \eta)$.
\end{proof}

\begin{proof}[Proof of Theorem \ref{t.ind}]
Let us first show (1) and (2b). We write $V_\ell=H^m(X_{\bar L},\Q_\ell)$.
By standard limit arguments, there exists a finitely generated sub-algebra
$R\subseteq L$ over $\F_p$ such that $X$ is defined over $B=\Spec(R)$: there
exists $f\colon \cX \to B$ proper smooth such that $X\simeq \cX\times_B
\eta$, where $\eta=\Spec(L)$. By Grothendieck trace formula, the system $(R
f_*\Qlb)_\ell$ on $B$ is compatible. Each $R^m f_*\Qlb$ is lisse and pure of
weight $m$. It follows that $(R^m f_*\Qlb)_\ell$ is compatible. By base
change, $(R^m f_*\Qlb)_{\bar \eta}\simeq V_\ell\otimes_{\Q_\ell}\Qlb$.
Applying Corollary \ref{c.val} to composition of the commutative square
\[\xymatrix{\eta\ar@{^{(}->}[d] \ar[r] & B\ar[d]\\
\Spec(\cO_L)\ar[r] & \Spec(\F_p)}
\]
and the closed immersion $\Spec(\F_p)\to \Spec(\cO_K)$, where $\cO_K$ is any
Henselian discrete valuation ring of residue field $\F_p$, we see that
$\tr(F,V_\ell)$ is a rational number independent of~$\ell$. In case (1), the
eigenvalues are roots of unity by Grothendieck's geometric local monodromy
theorem \cite[Variante 1.3]{SGA7-1}. In case (2b), the eigenvalues are
algebraic integers by a theorem of Ochiai \cite[Proposition A]{Ochiai}. It
follows that in both cases $\tr(F,V_\ell)$ is a rational integer independent
of $\ell$.

Part (2a) follows from (2b) and the monodromy weight conjecture, which is a
theorem of Terasoma \cite[Lemma 1.2]{Terasoma} and more generally Ito
\cite[Proposition 7.1]{Ito} in equal characteristic. Indeed, $\gr^M_i V $ is
pure of weight $m+i$, so that the characteristic polynomial of $F$ on
$\gr^M_i V$ can be extracted from the characteristic polynomial of $F$ on
$V$.
\end{proof}

The proof of Theorem \ref{t.ind} relies only on the special case of Theorem
\ref{t.main} with $S$ replaced by $\Spec(\F_p)$.

\begin{remark}\label{r.ind}\leavevmode
\begin{enumerate}
\item Ito's proof of the monodromy weight conjecture \cite{Ito} in equal
    characteristic and Grothendieck's proof of the geometric local
    monodromy theorem both use N\'eron's desingularization. For Theorem
    \ref{t.ind}, the reduction to the tame case is more involved and
    N\'eron's desingularization does not suffice.
\item Theorem \ref{t.ind} (2) implies that for each $F\in W^{\ge 0}(\bar
    \kappa/\kappa)$, $\tr(F,H^m(X_{\bar L},\Q_\ell)^I)$ is a rational
    integer independent of $\ell$. Here $I=I(\bar L/L)$ denotes the
    inertia group. The eigenvalues being algebraic integers, it suffices
    to show that the trace is in $\Q$ and independent of $\ell$. Consider
    the primitive parts of $V_\ell=H^m(X_{\bar L},\Q_\ell)$ defined by
    $P_i=\gr^M_{-i}\Ker(N)$ for $i\ge 0$. Here $N\colon V_\ell\to
    V_\ell(-1)$ is the logarithmic of the unipotent part of the local
    monodromy. By the identity $\gr_{-i}^M V_\ell=P_{i}\oplus \gr_{-i-2}^M
    V_\ell(-1)$, $i\ge 0$, the primitive parts $P_i$ and consequently
    $\Ker(N)$ are compatible. Moreover $V_\ell^I=\Ker(N)^I$ and there
    exists an open subgroup $U$ of $I$ acting trivially on $\Ker(N)$, so
    that $\tr(F,V_\ell^I)=\frac{1}{[I:U]}\sum_{F'}\tr(F',\Ker(N))\in \Q$
    is independent of $\ell$. Here $F'$ runs through liftings of $F$ in
    $\Gal(\bar L/L)/U$.
\item Theorem \ref{t.ind} (1) and (2b) hold in fact without the assumption
    that the valuation on $\cO_L$ is discrete. The proof that
    $\tr(F,H^m(X_{\bar L},\Q_\ell))$ is rational and independent of $\ell$
    is the same as above. For integrality, we apply Corollary \ref{c.int}
    below.
\end{enumerate}
\end{remark}

\section{Integrality along the boundary}\label{s.3}

Fix an integrally closed sub-ring $A$ of $\Qlb$. A typical example is the
integral closure of $\Z$ in $\Qlb$. Recall from \cite[Variantes 5.11,
5.13]{Zint} that a $\Qlb$-sheaf $\cF$ on a scheme $X$ of finite type over
$S$ is said to be \emph{integral} if for every $x\in \lvert X\rvert$, the
eigenvalues of $F\in W^{\ge 0}(\bar x/x)$ on $\cF_{\bar x}$ belong to $A$.
In this section, we study the integrality of integral sheaves on the
boundary.

\begin{defn}\label{d.int}
Let $\bar X$ be a normal scheme of finite type over $S$ and let $X$ be a
dense open subscheme. Let $\cF$ be a lisse $\Qlb$-sheaf on $X$. We say that
$\cF$ is \emph{integral on $\bar X$} if for every $x\in \lvert \bar X\rvert$
and every geometric point $\bar a\to X_{(x)}$, the eigenvalues of every
$F\in W^{\ge 0}(X_{(x)},\bar a)$ on $\cF_{\bar a}$ belong to $A$.
\end{defn}

We have the following analogues of Remarks \ref{r.triv} (1) and
\ref{r.closed}. A lisse $\Qlb$-sheaf $\cF$ on $X$ is integral on $X$ if and
only if it is integral. Moreover, in Definition \ref{d.int} we may restrict
to $x$ closed in $\bar X$.

\begin{remark}\label{r.fin}
Let $f\colon \bar X\to \bar Y$ be a finite surjective morphism of integral
normal schemes of finite type over $S$ and let $X\subseteq \bar X$,
$Y\subseteq \bar Y$ be nonempty open subschemes satisfying $f(X)\subseteq
Y$. Let $g\colon X\to Y$ be the restriction of $f$. Then a lisse
$\Qlb$-sheaf $\cF$ on $Y$ is integral on $\bar Y$ if and only if $g^*\cF$ is
integral on $\bar X$.

The ``only if'' part is obvious. For the ``if'' part, up to shrinking $X$
and $Y$ as in Remark \ref{r.triv} (2), we may assume that $g$ is the
composition of a universal homeomorphism with a finite \'etale morphism. In
this case, for every $x\in \bar X$, $\pi_1(X_{(x)},\bar a)$ is an open
subgroup of $\pi_1(Y_{(f(x))},f(\bar a))$, say of index $m$. Then for each
eigenvalue $\lambda$ of $F\in W^{\ge 0}(Y_{(f(x))},f(\bar a))$ acting on
$\cF_{\bar a}$, we have $\lambda^m\in A$ so that $\lambda\in A$.
\end{remark}

We have the following analogue of Theorem \ref{t.main}.

\begin{theorem}\label{t.int}
Let $X$ be a scheme of finite type over $S$ and let $\cF$ be an integral
lisse $\Qlb$-sheaf on $X$. Then there exists a finite stratification
$X=\bigcup_\alpha X_\alpha$ by normal subschemes such that each $X_\alpha$
admits a normal compactification $\bar X_\alpha$ over $S$ such that
$\cF|_{X_\alpha}$ is integral on $\bar X_{\alpha}$.
\end{theorem}

The theorem implies, by Remark \ref{r.comp}, that every compactification
$X_\alpha\subseteq \bar X'_\alpha$ is dominated by a normal compactification
$X_\alpha\subseteq\bar X''_\alpha$ such that $\cF|_{X_\alpha}$ is integral
on $\bar X''_\alpha$.

\begin{cor}
Let $\bar X$ be a projective smooth curve over $k$ or $K$ and let
$X\subseteq \bar X$ be a dense open subscheme. Then any integral lisse
$\Qlb$-sheaf $\cF$ on $X$ is integral on $\bar X$.
\end{cor}

The case of a curve over a finite field is a theorem of Deligne
\cite[Th\'eor\`eme 1.10.3]{WeilII}.

The proof of Theorem \ref{t.int} relies on the case $m=0$ of the following
theorem \cite[Th\'eor\`eme 2.5, Variantes 5.11, 5.13]{Zint}.

\begin{theorem}\label{t.intf}
Let $f\colon X\to Y$ be  a morphism of schemes of finite type over $S$. Let
$\cF$ be an integral $\Qlb$-sheaf on $X$. Then $R^m f_* \cF$ is integral for
all $m$.
\end{theorem}

The analogue for $R^mf_!$ was proved by Deligne and Esnault (\cite[XXI
Th\'eor\`eme 5.2.2]{SGA7-2}, \cite[Appendix, Theorem 0.2]{DE}). We refer to
\cite{Zint} for a description of the behavior of integral sheaves under
other operations.

\begin{cor}
Let $X$ be a scheme of finite type over $S$ and let $\cF$ be a lisse
$\Qlb$-sheaf on $X$. Assume that $\cF|_U$ is integral for some dense open
subscheme $U\subseteq X$. Then $\cF$ is integral.
\end{cor}

In the case $X$ of finite type over $k$, this was noted in \cite[Proposition
2.4]{comp}.

\begin{proof}
Up to replacing $X$ by its normalization, we may assume $X$ normal. Let
$j\colon U\to X$ be the open immersion. Then $\cF\simeq j_*(\cF|_U)$ in
integral by Theorem \ref{t.intf}.
\end{proof}

\begin{prop}\label{p.intNCD}
Let $\bar X$ be a regular scheme of finite type over $S$ and let $D$ be a
normal crossing divisor. Let $\cF$ be an integral lisse $\Qlb$-sheaf on
$X=\bar X-D$, tamely ramified on $\bar X$. Then $\cF$ is integral on~$\bar
X$. Moreover, $R^m j_* \cF(m)$ is integral for all $m$, where $j\colon X\to
\bar X$ is the open immersion.
\end{prop}

We will only need the first assertion. Some cases of the second assertion
were proved in \cite[Proposition 3.8, Variantes 5.11, 5.13]{Zind}.

\begin{proof}
We may assume that $D=\sum_{i\in I} D_i$ is a strict normal crossing divisor
with $D_i$ regular and defined globally by $t_i=0$, and $\cF$ is
$L$-ramified, where $L$ is the set of prime numbers invertible on $\bar X$.
We apply the construction of \cite[1.7.9]{WeilII}. For $J\subseteq I$, let
$D_J^*=\bigcap_{j\in J}D_j\cap \bigcap_{i\in J-I}(\bar X-D_i)$. For each
locally constant constructible sheaf of sets $\cG$ on $X$, $L$-ramified on
$\bar X$, there exists an integer $n$ invertible on $\bar X$ such that $\cG$
extends to $\cG'$ on the cover $\bar X[t_i^{1/n}]_{i\in I}$ of $\bar X$, and
we let $\cG[D_J^*]$ denote the restriction of $\cG'$ to $D_J^*$, which is
locally constant constructible. The action of $\mu_n^I$ on $\bar
X[t_i^{1/n}]_{i\in I}$ induces an action of $\mu_n^J$ on $\cG[D_J^*]$ (as
$\mu_n^J$ acts trivially on $D_J^*$). Extending this construction to
$\Qlb$-sheaves by taking limits, we obtain a lisse $\Qlb$-sheaf $\cF[D^*_J]$
on $D_J^*$ equipped with an action of $I_L^J$, where $I_L=\hat\Z_L(1)$.

Let us show that $\cF[D^*_J]$ is integral. For $J\subseteq J'\subseteq I$,
$\cF[D^*_{J'}]=\cF[D^*_J][D^*_{J'}]$. Thus by induction we may assume
$\#J=1$. Changing notation, it suffices to show that $\cF[D]$ is integral
for $D$ a regular divisor defined by $t=0$. By Grothendieck's arithmetic
local monodromy theorem applied to the Henselization of $\bar X$ at the
generic point of $D$, the action of $I_L=\hat\Z_L(1)$ on $\cF[D]$ is
quasi-unipotent. Up to replacing $\bar X$ by $\bar X[t^{1/n}]$, we may
assume that the action of $I_L$ on $\cF[D]$ is unipotent. Let $N\colon
\cF[D]\to \cF[D](-1)$ be the logarithm of the action of $I_L$ and let $M$ be
the local monodromy filtration on $\cF[D]$. Then $\Ker(N)=\cF[D]^{I_L}
\simeq (j_*\cF)|_D$, which is integral by Theorem \ref{t.intf}. Thus the
primitive parts $P_i=\gr^M_{-i} \Ker(N)$ are integral. It follows that
$\gr^M_i\cF[D]\simeq \bigoplus_{j} P_j(-\frac{j+i}{2})$ is integral. Here
$j$ runs through integers $j\ge \lvert i\rvert$ satisfying $j\equiv i\pmod
2$. Therefore $\cF[D]$ is integral.

Let $x\in D_J^*$. We have an exact sequence
\begin{equation}\label{e.exact}
1\to I_t\to \pi_1^t(X_{(x)},\bar a)\xrightarrow{r} \Gal(\bar x/x)\to 1.
\end{equation}
By Lemma \ref{l.mono}, there exists an open subgroup $V$ of $I_t$ acting
unipotently on $\cF_{\bar a}$. Assume $x\in \lvert D_J^*\rvert$ and let
$\cF'$ denote the semisimplification of $\cF|_{X_{(x)}}$. Then $V$ acts
trivially on $\cF'_{\bar a}$. The choice of a geometric point of $\lim_n
X_{(x)}[t_i^{1/n}]_{i\in I}$ above $\bar a$ gives a section $s$ of $r$ and
$\cF[D_J^*]_{x}$ corresponds to the action of $\Gal(\bar x/x)$ on $\cF_{\bar
a}$ via $s$. Then $U=V\cdot\Img(s)$ is an open subgroup of
$\pi_1^t(X_{(x)},\bar a)$. For $F\in U$, the eigenvalues of $F$ acting on
$\cF_{\bar a}$ are the same as the eigenvalues of $r(F)$ acting on
$\cF[D_J^*]_{\bar x}$, which belong to $A$ if $F\in W^{\ge 0}$. It follows
that $\cF$ is integral on $\bar X$.

For the second assertion of the proposition, note that the restriction of
$R^mj_*\cF$ to $D_J^*$ is $H^m(I_L^J,\cF[D_J^*])$. Since taking invariants
$H^0(I_L,-)=(-)^{I_L}$ and coinvariants $H^1(I_L,-)(1)\simeq (-)_{I_L}$
preserve integral sheaves, the same holds for $H^m(I_L^J,-)(m)\simeq
\bigoplus_{K} (-)^{I_L^{J-K}}_{I_L^K}$, where $K\subseteq J$ runs through
subsets of cardinality $m$.
\end{proof}

The rest of the proof of Theorem \ref{t.int} is similar to that of Theorem
\ref{t.main}. We proceed by Noetherian induction and reduce by Lemma
\ref{l.uh} to proving the following.

\begin{prop}
Let $X$ be an integral normal scheme separated of finite type over $S$ and
let $\cF$ be an integral lisse $\Qlb$-sheaf on $X$. Then there exist a
proper morphism $f\colon X'\to X$ with $X'$ connected normal inducing a
universal homeomorphism $f^{-1}(U)\to U$ for some nonempty open $U\subseteq
X$, and a normal compactification $X'\subseteq \bar X'$ over $S$ such that
$f^*\cF$ is integral on $\bar X'$.
\end{prop}

\begin{proof}
The proof is similar to that of Proposition \ref{p.key}, except that here we
do not need to work with stacks. We write
$\cF=(\cF_0)\otimes_{\overline{\Z_\ell}}\overline{\Q_\ell}$. There exists a
finite \'etale cover $Y\to X$, Galois of group $G$, such that
$\cF_0\otimes_{\overline{\Z_\ell}}
\overline{\Z_\ell}/\ell\overline{\Z_\ell}$ is constant. We apply the second
paragraph of the proof of Proposition \ref{p.key}. Since $\cF|_V$ is tamely
ramified on~$\bar Z$, $\cF|_V$ is integral on $\bar Z$ by Proposition
\ref{p.intNCD}. Thus, by Remark \ref{r.fin}, $f^*\cF$ is integral on $\bar
X'$, and the proposition follows.
\end{proof}

The same proof, with Proposition \ref{p.intNCD} replaced by Lemma
\ref{l.mono} applied to \eqref{e.exact}, yields the following result on
quasi-unipotence.

\begin{theorem}\label{t.unip}
Let $X$ be a scheme of finite type over $S$ and let $\cF$ be a lisse
$\Qlb$-sheaf on $X$. Then there exists a finite stratification
$X=\bigcup_\alpha X_\alpha$ by normal subschemes such that each $X_\alpha$
admits a normal compactification $\bar X_\alpha$ over $S$ such that for
every geometric point $\bar x \to \bar X$, and every geometric point $\bar
a\to (X_\alpha)_{(\bar x)}\colonequals (\bar X_\alpha)_{(\bar
x)}\times_{\bar X_\alpha} X_\alpha$, the action of $\pi_1((X_\alpha)_{(\bar
x)},\bar a)$ on $\cF_{\bar a}$ is quasi-unipotent. Here $(\bar
X_\alpha)_{(\bar x)}$ denotes the strict Henselization.
\end{theorem}

The analogues of Corollaries \ref{c.key} and \ref{c.val} hold with the same
proofs. Let us state the analogue of Corollary \ref{c.val}.

\begin{cor}\label{c.int}
Let $X$ be a scheme of finite type over $S$ and let $\cF$ be $\Qlb$-sheaf
on~$X$. Then $\cF$ is integral if and only if for every commutative square
\eqref{e.val} with $t$ quasi-finite over $S$, the eigenvalues of every $F\in
W^{\ge 0}(\bar \eta/\eta)$ acting on $\cF_{\bar \eta}$ belong to~$A$.
Moreover, if \eqref{e.val} is a commutative square with $V$ strictly
Henselian, then the action of $\Gal(\bar \eta/\eta)$ on $\cF_{\bar \eta}$ is
quasi-unipotent.
\end{cor}

\section{Ramified and decomposed parts of the fundamental group}\label{s.2}

In this section we give applications related to Vidal's ramified part of the
fundamental group \cite[Section 1.2]{Vidal2}. We show that compatible
systems have compatible ramification (Corollary \ref{c.ram}), and
consequently, their reductions have compatible wild ramification (Corollary
\ref{c.d}).

Let us first review the definition of the ramified part of the fundamental
group. Let $T$ be the spectrum of an excellent Henselian discrete valuation
ring of residue characteristic exponent $p\ge 1$.

\begin{defn}[Vidal]\label{d.Vidal}
Let $X$ be an integral normal scheme separated of finite type over $T$ and
let $\bar a$ be a geometric generic point of $X$. Let $X\subseteq \bar X$ be
a normal compactification over $T$. Let $\bar x \to \bar X$ be a geometric
point above $x\in \bar X$ and let $\bar X_{(\bar x)}$ denote the strict
Henselization. Let $\bar b\to X_{(\bar x)}\colonequals X\times_{\bar X} \bar
X_{(\bar x)}$ be a geometric point above $\bar a$. We define the following
closed subsets of $\pi_1(X,\bar a)$, each of which is a union of subgroups:
\begin{itemize}
\item The subgroup $E_{X,\bar X, x,\bar b}=\Img(\pi_1(X_{(\bar x)},\bar
    b)\to \pi_1(X,\bar a))$. See Remark \ref{r.val} (1) below for the
    justification of the subscript $x$ instead of $\bar x$.
\item $E_{X,\bar X}$ is the closure of $\bigcup_{x,\bar b} E_{X,\bar
    X,x,\bar b}$, where $x$ runs through points of $\bar X$ and $\bar b$
    runs through geometric points above $\bar a$.
\item The \emph{ramified part} $E_{X/T}=\bigcap_{\bar X} E_{X,\bar X}$,
    where $\bar X$ runs through normal compactifications of $X$ over $T$.
\end{itemize}
\end{defn}

The subsets $E_{X,\bar X}$ and $E_{X/T}$ are stable under conjugation.

\begin{remark}\leavevmode\label{r.val}
\begin{enumerate}
\item We have a short exact sequence
\[
1\to\pi_1(
    X_{(\bar x)},\bar b)\xrightarrow{i} \pi_1(
     X_{(x)},\bar b)\xrightarrow{\rho} \pi_1(\bar X_{(x)},\bar b)\to 1,
\]
where $\pi_1(\bar X_{(x)},\bar b)\simeq \Gal(\bar x/x)$. The image of $i$
depends on $\bar x$ only via $x$ and depends on $\bar b$ as a geometric
point of $X_{(x)}$.

\item For any specialization $\bar x\to X_{(\bar y)}$, we have $E_{X,\bar
    X,x,\bar b}\subseteq  E_{X,\bar X,y,\bar b}$. Thus in the definition
    of $E_{X,\bar X}$, we may restrict to closed points $x\in \bar X$.

\item It follows from Gabber's valuative criterion \cite[Section
    6.1]{Vidal2} that for any finite stratification $X=\bigcup_\alpha
    X_\alpha$ into integral normal subschemes, $E_{X/T}$ is the closure of
    $\bigcup_{\alpha,\gamma_\alpha}\gamma_\alpha(E_{X_\alpha/T})$, where
    $\gamma_\alpha$ runs through paths from a geometric generic point
    $\bar a_\alpha\to X_\alpha$ to $\bar a\to X$.
\end{enumerate}
\end{remark}

Recall that $S$ is the spectrum of an excellent Henselian discrete valuation
ring of \emph{finite} residue field.

\begin{defn}
Let $X$ be a scheme of finite type over $S$. We say that a system $(L_i)\in
\prod_{i\in I}K(X,\Qlib)$ has \emph{compatible ramification} if for every
separated integral normal subscheme $Y\subseteq X$, $(\tr(g,(L_i)_{\bar
a}))_{i\in I_Y}$ is compatible for all $g\in E_{Y/T}$. Here $\bar a$ is a
geometric generic point of $Y$, $I_Y\subseteq I$ is the subset of $i$ such
that $L_i|_Y$ is in $K_\lisse(Y,\Qlib)$.
\end{defn}

\begin{remark}
Let $X$ be an integral normal scheme separated of finite type over $S$ and
let $\cF$ be a lisse $\Qlb$-sheaf on $X$. Then the action of $E_{X/S}$ on
$\cF_{\bar a}$ is quasi-unipotent. This follows from Theorem \ref{t.unip},
Remark \ref{r.val} (3), and the fact that quasi-unipotent matrices form a
closed subset of $\GL_n(\Qlb)$ (Remark \ref{r.unip}).
\end{remark}

Combining this with Gabber's valuative criterion \cite[Section 6.1]{Vidal2},
we obtain the following valuative criterion for compatible ramification.

\begin{lemma}\label{l.val}
Let $X$ be a scheme of finite type over $S$. Then $(L_i)\in \prod_{i\in
I}K(X,\Qlib)$ has compatible ramification if and only if for every
commutative square \eqref{e.val} with $V$ strictly Henselian,
$\tr(F,(L_i)_{\bar\eta})_{i\in I}$ is compatible for all $F\in \Gal(\bar
\eta/\eta)$.
\end{lemma}

\begin{proof}
We may assume $I$ finite, $X$ integral normal separated, and
$L_i=[\cF_i]-[\cG_i]$ with $\cF_i$ and $\cG_i$ lisse, respectively of rank
$m_i$ and $n_i$. Consider the continuous map
\[\sigma\colon E_{X/S}\to
C=\prod_{i\in I} (P^{m_i,\qu}_{\ell_i}\times P^{n_i,\qu}_{\ell_i})
\]
carrying $g$ to $(\det(g-T\cdot 1,(\cF_i)_{\bar a}), \det(g-T\cdot
1,(\cG_i)_{\bar a}))$, where $P^{r,\qu}_{\ell_i}$ is as in Remark
\ref{r.unip}. Gabber's criterion says that $E_{X/S}$ is the closure of the
union of the images of $\Gal(\bar \eta/\eta)$. Since $C$ is discrete,
$\sigma(E_{X/S})$ is the union of the images of $\Gal(\bar\eta/\eta)$.
\end{proof}

Corollary \ref{c.val} (2) now takes the following form.

\begin{cor}\label{c.ram}
Let $X$ be a scheme of finite type over $S$. Then any compatible system
$(L_i)\in \prod_{i\in I}K(X,\Qlib)$ has compatible ramification.
\end{cor}

Let $P$ be a set of prime numbers. Given a profinite group $G$, we let
$G^P\subseteq G$ denote the subset of elements $g$ such that all prime
factors of the supernatural order of $g$ are contained in $P$. Note that
$G^P$ is a closed subset stable under conjugation, union of subgroups of
$G$. For a continuous homomorphism of profinite groups $\alpha\colon G\to
H$, we have $\alpha(G^P)=\alpha(G)\cap H^P$.

We write $(p)=\{p\}$ for $p>1$ and $(p)=\emptyset$ for $p=1$. Then $G^{(p)}$
is the union of the $p$-Sylow subgroups of $G$.

\begin{notation}
In the situation of Definition \ref{d.Vidal}, for $G=\pi_1(X,\bar a)$, we
write
\[E^P_{X,\bar X,x,\bar b}=E_{X,\bar X,x,\bar b}\cap G^P, \quad E^P_{X,\bar X}=E_{X,\bar X}\cap G^P, \quad E^P_{X/T}=E_{X/T}\cap G^P.\]
\end{notation}

\begin{remark}
Alternatively we can define these subsets as follows:
\begin{itemize}
\item $E^P_{X,\bar X, x,\bar b}=\Img(\pi_1(X_{(\bar x)},\bar b)^P\to
    \pi_1(X,\bar a))$.
\item $E^P_{X,\bar X}$ is the closure of $\bigcup_{x,\bar b}E^P_{X,\bar
    X,x,\bar b}$, where $x$ runs through points of $\bar X$ and $\bar b$
    runs through geometric points above $\bar a$.
\item $E^P_{X/T}=\bigcap_{\bar X} E^P_{X,\bar X}$, where $\bar X$ runs
    through normal compactifications of $X$ over~$T$.
\end{itemize}
\end{remark}

For $P=(p)$, $E^{(p)}_{X/T}$ is called the \emph{wildly ramified part} of
the fundamental group and was defined by Vidal \cite[2.1]{Vidal1}. Our
notation differs from that of Vidal, who writes $E'_{X/T}$ and $E_{X/T}$ for
our $E_{X/T}$ and $E^{(p)}_{X/T}$, respectively.

Next we define compatible $P$-ramification for systems of $\Flb$-sheaves,
where $\Flb$ denotes an algebraic closure of $\F_\ell$. For a profinite
group $G$, an element $g\in G$ that is $\ell$-regular (namely, of
supernatural order prime to $\ell$), and a virtual $\Flb$-representation $M$
of $G$, the Brauer trace is defined by $\tr^\Br(g,M)=\sum_\lambda
[\lambda]$, where $\lambda$ runs through eigenvalues of $g$ acting on $M$
(with multiplicities), and $[\lambda]$ denotes the Teichm\"uller lift. Note
that $\tr^\Br(g,M)$ is a sum of roots of unity (of order prime to $\ell$) in
$\Qlb$.

Let $X$ be a scheme of finite type over $T$. Let $K(X,\Flb)$ denote the
Grothendieck group of constructible $\Flb$-sheaves.  Recall that we fixed
field embeddings $\iota_i\colon Q\to {\Qlib}$ for $i\in I$.

\begin{defn}\label{d.comp}
Assume that $P$ does not contain any $\ell_i$. We say that a system
$(L_i)\in \prod_{i\in I}K(X,\Flib)$ has \emph{compatible $P$-ramification}
if for every separated integral normal subscheme $Y\subseteq X$,
$(\tr^\Br(g,(L_i)_{\bar a}))_{i\in I_Y}$ is compatible for all $g\in
E^P_{Y/T}$. Here $\bar a$ is a geometric generic point of $Y$, $I_Y\subseteq
I$ is the subset of $i$ such that $L_i|_Y$ is in $K_\lisse(Y,\Flib)$. We say
that $(L_i)_{i\in I}$ has \emph{compatible wild ramification} if it has
compatible $(p)$-ramification.
\end{defn}

In the special case $\ell_i=\ell$ and $Q=\Qlb$, one recovers the notion of
same wild ramification of Deligne \cite{IllBrauer} and Vidal \cite{Vidal1}.
A weaker condition was recently studied by Saito and Yatagawa (\cite{SY},
\cite{Yatagawa}). In \cite{Guo} Guo shows that systems of compatible wild
ramification in the sense of Definition \ref{d.comp} are preserved by
Grothendieck's six operations and duality.

Gabber's valuative criterion \cite[Section 6.1]{Vidal2} implies the
following valuative criterion for compatible $P$-ramification.

\begin{lemma}\label{l.Gabber}
Let $X$ be a scheme of finite type over $T$. Then $(L_i)\in \prod_{i\in
I}K(X,\Flib)$ has compatible $P$-ramification if and only if for every
commutative square
\[
\xymatrix{\eta\ar@{^{(}->}[d] \ar[r] & X\ar[d]\\
\Spec(\cO_L)\ar[r] & T}
\]
where $\cO_L$ is a strictly Henselian valuation ring of fraction field $L$
and $\eta=\Spec(L)$, $(\tr^\Br(g,(L_i)_{\bar \eta}))$ is compatible for all
$g\in \Gal(\bar \eta/\eta)^P$. Here $\bar \eta \to \eta$ is a geometric
point.
\end{lemma}

\begin{remark}
Let $X$ be an integral normal scheme separated of finite type over~$T$ and
let $(L_i)\in \prod_{i\in I} K_\lisse(X,\Flib)$ with $I$ finite. Then
$(L_i)_{i\in I}$ has compatible $P$-ramification if and only if there exists
a normal compactification $X\subseteq \bar X$ over $T$ such that
$(\tr^\Br(g,(L_i)_{\bar a}))_{i\in I}$ is compatible for all $g\in
E^P_{X,\bar X}$.

Indeed, for any finite quotient $G$ of $\pi_1(X,\bar a)$, if we let
$E^P_{X,\bar X}(G)$ and $E^P_{X/T}(G)$ denote respectively the images of
$E^P_{X,\bar X}$ and $E^P_{X/T}$ in $G$, then we have
$E^P_{X/T}(G)=\bigcap_{\bar X} E^P_{X,\bar X}(G)$ by Lemma \ref{l.top}
below, and it follows that $E^P_{X/T}(G)=E^P_{X,\bar X}(G)$ for some $\bar
X$. Here we used the fact that any pair of normal compactifications is
dominated by a third one (cf.\ Remark \ref{r.comp}).
\end{remark}

\begin{lemma}\label{l.top}
Let $\Pi$ be a topological space and let $B$ be a downward directed set of
closed subsets of $\Pi$: for $E_1,E_2\in B$, there exists $E\in B$ such that
$E\subseteq E_1\cap E_2$. Let $\sigma\colon \Pi\to C$ be a map such that all
fibers are compact. Then $\sigma(\bigcap_{E\in B}E)=\bigcap_{E\in
B}\sigma(E)$.
\end{lemma}

\begin{proof}
We have $\sigma(\bigcap_{E\in B}E)\subseteq \bigcap_{E\in B}\sigma(E)$.
Conversely, let $g\in C-\sigma(\bigcap_{E\in B}E)$. Then $\sigma^{-1}(g)\cap
\bigcap_{E\in B}E=\emptyset$. Since $\sigma^{-1}(g)$ is compact, there exist
$E_1,\dots,E_n\in B$ such that $\sigma^{-1}(g)\cap \bigcap_{i=1}^n
E_i=\emptyset$. Since $B$ is downward directed, there exists $E\in B$ such
that $\sigma^{-1}(g)\cap E=\emptyset$. In other words, $g\in C-\bigcap_{E\in
B}\sigma(E)$.
\end{proof}

\begin{defn}
Let $E_\lambda$ by a finite extension of $\Q_\ell$, of ring of integers
$\cO_\lambda$ and residue field $\F_\lambda$. Consider the composition
\[K(X,E_\lambda)\xrightarrow[\sim]{(j^*)^{-1}} K(X,\cO_\lambda)
\xrightarrow{i^*} K(X,\F_\lambda),
\]
where $j^*$ is given by $-\otimes_{\cO_\lambda} E_\lambda$ and $i^*$ is
given by $-\otimes^L_{\cO_\lambda} \F_\lambda$. By \cite[Proposition
9.4]{six}, $j^*$ is an isomorphism and $i^*$ is a surjection. Taking
colimit, we get the \emph{decomposition map}
\[d_X\colon K(X,\Qlb)\to K(X,\Flb),\]
which is a surjection.
\end{defn}

It follows from the definition that if $(L_i)\in \prod_{i\in I}K(X,\Qlib)$
has compatible ramification, then $(d_X(L_i))\in \prod_{i\in I}K(X,\Flib)$
has compatible $P$-ramification. Thus Corollary \ref{c.ram} implies the
following.

\begin{cor}\label{c.d}
Let $X$ be a scheme of finite type over $S$. Let $(L_i)\in \prod_{i\in I}
K(X,\Qlib)$ be a compatible system. Then $(d_X(L_i))\in \prod_{i\in
I}K(X,\Flib)$ has compatible $P$-ramification, where $P$ is the set of
primes not equal to any $\ell_i$. In particular, $(d_X(L_i))_{i\in I}$ has
compatible wild ramification.
\end{cor}

Next we define the decomposed part of the fundamental group. The first three
steps of the definition are analogous to Definition \ref{d.Vidal}, with
inertia groups (associated to strict Henselizations) replaced by
decomposition groups (associated to Henselizations).

\begin{defn}
Let $X$ be an integral normal scheme separated of finite type over $T$ and
let $\bar a$ be a geometric generic point of $X$. Let $X\subseteq \bar X$ be
a normal compactification over $T$. Let $x\in \bar X$ be a point and let
$\bar b\to X_{(x)}$ be a geometric point above $\bar a$. Let
$X=\bigcup_{\alpha} X_\alpha$ be a finite stratification of $X$ into
integral normal subschemes. We define the following closed subsets of
$\pi_1(X,\bar a)$, each of which is a union of subgroups:
\begin{itemize}
\item The subgroup $D_{X,\bar X,x,\bar b}=\Img(\pi_1(X_{(x)},\bar b)\to
    \pi_1(X,\bar a))$.
\item $D_{X,\bar X}$ is the closure of $\bigcup_{x,\bar b} D_{X,\bar X,
    x,\bar b}$, where $x$ runs through \emph{locally closed} points of
    $\bar X$ and $\bar b$ runs through geometric points above $\bar a$.
\item $D^\naive_{X/T}=\bigcap_{\bar X} D_{X,\bar X}$, where $\bar X$ runs
    through normal compactifications of $X$ over $T$.
\item $D_{X/T,(X_\alpha)}$ is the closure of
    $\bigcup_{\alpha,\gamma_\alpha}\gamma_\alpha(D^\naive_{X_\alpha/T})$,
    where $\gamma_\alpha$ runs through paths from a geometric generic
    point $\bar a_\alpha\to X_\alpha$ to $\bar a\to X$.
\item The \emph{decomposed part} $D_{X/T}=\bigcap D_{X/T,(X_\alpha)}$,
    where $(X_\alpha)$ runs through finite stratifications of $X$ into
    integral normal
subschemes.
\end{itemize}
\end{defn}

Except for $D_{X,\bar X,x,\bar b}$, the above subsets are stable under
conjugation.

The definition is functorial in an obvious sense, which we specify for
$D^\naive_{X/T}$ and $D_{X/T}$. Given a morphism $f\colon X\to Y$ of
integral normal schemes of finite type over $T$ and a path $\gamma$ from
$\bar a\to X$ to a geometric generic point $\bar a'\to Y$, the induced
homomorphism $\gamma\colon \pi_1(X,\bar a)\to \pi_1(Y,\bar a')$ satisfies
$\gamma(D^\naive_{X/T})\subseteq D^\naive_{Y/T}$ and
$\gamma(D_{X/T})\subseteq D_{Y/T}$. We have $D_{X/T}\subseteq
D_{X/T,(X_\alpha)}\subseteq D^\naive_{X/T}$, where the second inclusion
follows from the functoriality of $D^\naive_{X/T}$.

\begin{remark}\leavevmode\label{r.D}
\begin{enumerate}
\item  In the absence of a valuative criterion, we performed the last two
    steps in the definition to ensure that for any finite stratification
    $(X_\alpha)$ of $X$ into integral normal subschemes, $D_{X/T}$ is the
    closure of
    $\bigcup_{\alpha,\gamma_\alpha}\gamma_\alpha(D_{X_\alpha/T})$, where
    $\gamma_\alpha$ runs through paths from a geometric generic point
    $\bar a_\alpha\to X_\alpha$ to $\bar a\to X$.
\item If $a$ denotes the generic point of $X$, then $D_{X,\bar X,a,\bar
    a}=\pi_1(X,\bar a)$. On the other hand, for $x\in \bar X$ locally
    closed, the closure $\overline{\{x\}}$ is finite over $T$ and if we
    let $y$ denote the closed point of $\overline{\{x\}}$, then we have a
    canonical morphism $X_{(x)} \to X_{(y)}$ as in Remark \ref{r.closed},
    so that $D_{X,\bar X,x,\bar b}\subseteq D_{X,\bar X,y,\bar b}$. Thus,
    in the definition of $D_{X,\bar X}$, we may restrict to closed points
    $x\in \bar X$.
\item For $x\in \bar X$ closed, the exact sequence in Remark \ref{r.val}
    (1) induces an exact sequence
    \[1\to E_{X,\bar X,x,\bar b} \to D_{X,\bar X,x,\bar b}\to \pi_1(\bar X_{(x)},\bar b) \to 1.\]
    Indeed, in the commutative square
\[\xymatrix{\pi_1(X_{(x)},\bar b)\ar[r]^\rho\ar[d]_{\tau} & \pi_1(\bar X_{(x)},\bar b)\ar[d]^\iota\\
\pi_1(X,\bar a)\ar[r]^\sigma & \pi_1(T,\bar a),}
\]
$\iota$ is an injection, so that $\Ker(\tau)\subseteq \Ker(\rho)$. Let
$K=\Ker(\sigma)$. Then
\[E_{X,\bar X,x,\bar b}=K\cap D_{X,\bar X,x,\bar b},\quad E_{X,\bar X}\subseteq K\cap D_{X,\bar X}, \quad E_{X/T}\subseteq K\cap D_{X/T}.\]
\item Assume $T=S$. If $X_k$ is geometrically unibranch, then $D_{X/S}$
    contains the image of $\pi_1(X_k)$. Indeed, $D_{X_k/S}\subseteq
    \pi_1(X_k)$ contains the Frobenius element at every $x\in \lvert X_k
    \rvert$, so that $D_{X_k/S}=\pi_1(X_k)$ in this case by Chebotarev's
    density theorem. If moreover $X$ is proper over $S$ so that
    $\pi_1(X_k)\simeq \pi_1(X)$ \cite[XII Th\'eor\`eme 5.9]{SGA4-3}, then
    $D_{X/S}=\pi_1(X,\bar a)$. The equality does not hold in general, even
    for $X$ proper over $S$.
\end{enumerate}
\end{remark}

Theorem \ref{t.main} implies the following density result.

\begin{cor}\label{t.dense}
Let $X$ be an integral normal scheme separated of finite type over~$S$. Then
$D_{X/S}$ is the closure of $\bigcup_{\bar x,\gamma}\gamma(W(\bar x/x))$,
where $\bar x$ runs through geometric points of $X$ above $x\in \lvert
X\rvert$ and $\gamma$ runs through paths from $\bar x\to x$ to $\bar a\to
X$.
\end{cor}

In the corollary, we may replace $W(\bar x/x)$ by $W^{\ge
    N}(\bar x/x)$, which is a dense subset of $W(\bar x/x)$ for the
    profinite topology. Moreover, we may restrict to closed points $x\in X$ as in Remark \ref{r.closed}.

\begin{proof}
Let $C$ be the closure of $\bigcup_{\bar x,\gamma}\gamma(W(\bar x/x))$. We
have $C\subseteq D_{X/S}$. Let $G$ be a finite quotient of $\pi_1(X,\bar a)$
and let $C(G)$ and $D(G)$ denote respectively the images of $C$ and
$D_{X/S}$ in $G$. It suffices to show that for any pair of $\Qlb$-characters
$\chi$ and $\chi'$ of $G$ satisfying $\chi|_{C(G)}=\chi'|_{C(G)}$, we have
$\chi|_{D(G)}=\chi'|_{D(G)}$. Let $\cF$ and $\cF'$ be the corresponding
lisse $\Qlb$-sheaves on $X$. Then $\cF$ and $\cF'$ are compatible (for
$Q=\Qlb$). We apply Theorem \ref{t.main}. Since $D_{X/S}\subseteq
\bigcup_{\alpha,\gamma_\alpha} \gamma_\alpha(D^\naive_{X_\alpha/S})\subseteq
\bigcup_{\alpha,\gamma_\alpha} \gamma_\alpha(D_{X_\alpha,\bar X_\alpha})$,
every $g\in D(G)$ is in the image of some $\pi_1((X_\alpha)_{(x)},\bar b)$
for $x$ closed in $\bar X_\alpha$, which equals the image of
$W((X_\alpha)_{(x)},\bar b)$. Thus $\chi|_{D(G)}=\chi'|_{D(G)}$.
\end{proof}

\end{document}